\documentclass[12pt]{amsart}
\usepackage{amsmath,amssymb,amsthm,amscd}

\newtheorem{Thm}{Theorem}
\newtheorem{Lem}{Lemma}
\newtheorem{Prop}{Proposition}
\newtheorem{Cor}{Corollary}

\newtheorem{Rque}{Remark}

\def\Hilb{\mathop{\rm Hilb}\nolimits}

\def\dim{\mathop{\rm dim}\nolimits}

\def\min{\mathop{\rm min}\nolimits}
\def\Mor{\mathop{\rm Mor}\nolimits}

\def\nit{{\mathbb N}}

\def\zit{{\mathbb Z}}
\def\pit{{\mathbb P}}
\def\0{{\mathcal O}}

\numberwithin{equation}{section}

\newcommand{\Pic}{\operatorname{Pic}}

\renewcommand{\O}{{\mathcal O}}

\newcommand{\I}{{\mathcal I}}

\newcommand{\Z}{{\mathbb Z}}

\newcommand{\p}{{\mathbb P}}

\newcommand{\defect}{\operatorname{def}}

\title{Inductive characterizations of hyperquadrics}
\author{Baohua Fu}
\begin{document}
\maketitle
\begin{abstract}
We give two characterizations of hyperquadrics: one as
non-degenerate smooth projective varieties swept out by large
dimensional quadric subvarieties  passing through a point; the
other as $LQEL$-manifolds with large secant defects.
\end{abstract}

\section{Introduction}
 We work over an algebraically closed field of
characteristic zero. In \cite{Ein}, Ein proved that if $X$ is an
$n$-dimensional smooth projective variety containing an $m$-plane
$\Pi_0$ whose normal bundle is trivial, with $m \geq n/2+1$, then
there exists a smooth projective variety $Y$ and  a vector bundle
$E$ over $Y$ such that $X \simeq \pit(E)$ and  $\Pi_0$ is a fiber
of $X \to Y$. The bound on $m$ was improved to $m \geq n/2$ by
Wi\'sniewski in \cite{Wis}. Later on, Sato \cite{Sat} studied
projective smooth $n$-folds $X$ swept out by $m$-dimensional
linear subspaces, i.e. through every point of $X$, there passes
through an $m$-dimensional linear subspace. If  $m \geq n/2$, he
proved that either $X$ is a projective bundle as above or $m=n/2$.
In the latter case, $X$ is either a smooth hyperquadric or the
Grassmanian variety parametrizing lines in $\pit^{m+1}$.

A natural problem is to extend these results to the case where
linear subspaces are replaced by quadric hypersurfaces. In this
paper, we will  consider a smooth projective non-degenerate
variety  $X \subsetneq \pit^N$ of dimension $n$, which is swept
out by $m$-dimensional irreducible hyperquadrics passing through a
point (for the precise definition see section 3). Examples of such
varieties include Severi varieties (see \cite{Zak}), or more
generally LQEL manifolds of positive secant defect(see section 2
below). As it turns out, the number $m$ is closely related to the
secant defect of $X$, which makes it hard to construct examples
with big $m$.

Our main theorem is to  show (cf. Thm. \ref{main}) that if
$m>[n/2]+1$, then $N=n+1$ and $X$ is itself a hyperquadric. This
gives a substantial improvement to the Main theorem 0.2 of
\cite{KS}, where the same claim is proved under the assumption
that a general hyperquadric in the family is smooth and that  $m
\geq 3n/5+1$. Our proof here, based on ideas contained in
\cite{IR2}
 and \cite{Rus}, is much simpler and is completely
different from that in \cite{KS}. However, we should point out
that a more general result, without assuming the quadric subspaces
pass all through a fixed point,  is proven in \cite{KS}.

The same idea of proof, combined with the  Divisibility Theorem of
\cite{Rus}, allows us to prove (cf. Corollary \ref{boundelta})
that for an $n$-dimensional $LQEL$-manifold, either it is a
hyperquadric or its secant defect is no bigger than
$\frac{n+8}{3}$. This improves Corollary 0.11, 0.14 of \cite{KS}.
It also gives positive support to the general believing that
hyperquadrics are the only LQEL manifolds with large secant
defects.

\section{Preliminaries}

Let $\delta=\delta(X)=2n+1-\dim(SX)$  be the {\it secant defect}
of a non-degenerate $n$-dimensional  variety $X\subset\p^N$, where
$$SX=\overline{\textstyle{\bigcup\limits _{{\scriptstyle x\neq y }\atop {\scriptstyle x,y\in
X}}}\langle x,y\rangle }\subseteq\p^N $$
 is the {\em secant variety} of $X\subset\p^N$.

 Recall(\cite{KS}, \cite{IR}) that a smooth
irreducible non-degenerate projective variety $Z\subset\p^N$ is
said to be {\em conically connected} (CC for short) if through two
general points there passes an irreducible conic contained in $Z$.
Such varieties have been  studied and classified in \cite{IR} and
\cite{IR2}.

We begin with a simple  but very useful remark, which is probably
well known but  we were not able to find a  reference.
\begin{Lem}\label{cclem}
Let $X\subset\p^N$ be a smooth projective variety and  let $z \in
X$ be a  point. If  there exists a family of  smooth rational
curves of degree $d$ on $X$  passing through $z$ and covering $X$,
then through two general points $x,y\in X $ there passes such a
curve.

In particular,  if $d=1$, then $X\subset\p^N$ is a linearly
embedded $\p^n$. If $d=2$ and if $X\subset\p^N$ is non-degenerate,
then $X\subset\p^N$ is conically connected.
\end{Lem}
\begin{proof}

By Theorem II.3.11 \cite{Kol}, there exists finitely  many closed
subvarieties (depending on $z$) $V_i \varsubsetneq X$,
$i=1,\ldots, l$, such that for any nonconstant morphism $f: \pit^1
\to X$  with  $f(0)=z$,   $\deg(f_*(\p^1))=d$ and with $f(\p^1)
\nsubseteq \cup_{i=1}^l V_i$, we have $f^* T_X$ is ample.
  Now take a general point $x \in X\setminus\cup_{i=1}^l V_i$
and a smooth  rational curve  $C\subset X$ of degree $d$ passing
through  $x$ and $z$. The above result implies that
$f^*T_X=T_X|_C$ is ample and hence that $N_{C|X}$ is ample. Thus
there exists a unique irreducible component $W_x$ of the Hilbert
schemes of rational curves of degree $d$ contained in $X$ and
passing through $x$ containing $[C]$. Since $N_{C|X}$ is ample, it
is well known that
 deformations of $C$ parametrized by $W_x$  cover $X$.
  Therefore given a general point $y
\in X$, we can find a smooth rational curve of degree $d$
contained in $X$  joining $x$ and $y$,  proving the first part of
the assertion. To conclude the proof it is sufficient to recall
that linear subspaces of $\p^N$ are the unique irreducible
subvarieties containing  the line through two general points of
itself.
\end{proof}

The following general result on CC-manifolds is proved  in
\cite{IR2}.
\medskip

\begin{Prop}\label{LQELCC}{\rm (\cite[Prop. 3.2]{IR2})}
 Let $X\subset\p^N$ be a CC-manifold and let $C=C_{x,y}$ be
 a general conic through the general points $x,y\in X$.
Then

$$n+\delta(X)\geq -K_X\cdot C\geq n+1.$$

If moreover $\delta(X)\geq 3$, then $X\subset\p^N$ is a Fano
manifold with $\Pic(X)\simeq\mathbb{Z}\langle \O_X(1)\rangle$,
whose index $i(X)$ satisfies
\[\frac{n+\delta(X)}{2}\geq i(X)\geq \frac{n+1}{2}.\]
\end{Prop}
\medskip

Now consider a smooth projective variety $X \subset \pit^N$. For a
general point $x \in X$,  let $Y_x$ be the Hilbert scheme of lines
on $X\subset\p^N$ passing through $x$, which can be naturally
regarded as a sub-variety in $\pit((t_x X)^*)=\p^{n-1}$, where
$t_xX$ is the affine tangent space to $X$ at $x$. The variety
$Y_x$ is  the first instance of the so-called {\em variety of
minimal rational tangents},  introduced and extensively studied by
Hwang and Mok (see \cite{Hwa} and the references therein). When
$X\subset\p^N$ is a Fano manifold with $\Pic(X)\simeq\mathbb
Z\langle \O(1)\rangle$, there exists a deep connection between
geometrical properties of $Y_x\subset\p^{n-1}$ and the index of
$X$. The following result contained  in \cite[Prop. 2.4]{IR} is
essentially due to Hwang and Kebekus, cf. \cite[Th. 3.14]{HK}.
\medskip

\begin{Prop}{\rm (\cite[Th. 3.14]{HK} and \cite[Prop. 2.4]{IR})}\label{largeindex}
 Let  $X\subset\p^N$ be a Fano manifold with
 $\Pic(X)\simeq \Z \langle H\rangle$ and $-K_X = i(X)H$,
 $H$ being the hyperplane section and $i(X)$ the index of $X$.
\begin{enumerate}
\item[(i)] If $i(X)>\frac{n+1}{2}$, then $X\subset\p^N$ is ruled
by lines and for general $x\in X$ the Hilbert scheme of lines
through $x$, $Y_x\subset \p((\mathbf{T}_xX)^*)=\p^{n-1}$, is
smooth. If $i(X)\geq\frac{n+3}{2}$, $Y_x$ is also irreducible.
\item[(ii)] If $i(X)\geq\frac{n+3}{2}$ and  $SY_x=\p^{n-1}$, then
$X\subset\p^N$ is a CC-manifold. \item[(iii)] If
$i(X)>\frac{2n}{3}$, then $X\subset\p^N$ is a CC-manifold with
$\delta(X)> \frac{n}{3}$ and such that $SY_x=\p^{n-1}$.
\end{enumerate}
\end{Prop}
\medskip

Recall that (cf. \cite{KS}, \cite{Rus}, \cite{IR2}) a smooth
irreducible non-degenerate variety $X \subset \p^N$ is said to be
a {\em local quadratic entry locus manifold of type $\delta \geq
0$} ($LQEL$-manifold for short) if for general $x, y \in X$
distinct points, there exists a hyperquadric of dimension $\delta
= \delta(X)$ contained in $X$ and passing through $x, y$. By
definition, a LQEL manifold of positive secant defect is conically
connected, but the converse is not true. For example, a smooth
cubic hypersurface $X\subset\p^{n+1}$  with $n\geq 3$ is conically
connected but not a LQEL-manifold.
 Severi varieties and Scorza varieties are
basic examples of LQEL manifolds (\cite{Zak}).

A systematic study of LQEL manifolds has been succesively carried
out by Russo in \cite{Rus}, in particular, the following
remarkable theorem has been proved in \cite{Rus}.
\medskip

\begin{Thm}{\rm (\cite[Th. 2.8] {Rus})}\label{induction}
 For an $n$-dimensional $LQEL$-manifold $X\subset\p^N$ of type $\delta \geq 3$,  let $x \in X$ be a
general point and let $Y_x\subset\p^{n-1}$ be the Hilbert scheme
of lines on $X$ passing through $x$. Then $Y_x\subset\p^{n-1}$ is
a $LQEL$-manifold of type $\delta-2$, of dimension
$(n+\delta)/2-2$ and such that $SY_x=\p^{n-1}$.
Let $\delta=2r_X+1$, or $\delta=2r_X+2$. Then $2^{r_X}$ divides
$n-\delta$.
\end{Thm}
\medskip

\section{Varieties swept out by hyperquadrics}
Through out this section, let $X \subsetneq \pit^N$ be an
$n$-dimensional non-degenerate projective smooth variety which
satisfies the following two conditions:
\begin{enumerate}
\item[i)] through a general point $x \in X$, there passes an
irreducible reduced $m$-dimensional quadric $Q_x\subset
X\subset\p^N$, where $m$ is a fixed natural number (i.e. the
linear span $<Q_x>$ of $Q_x$ in $\p^N$ is a linear subspace of
dimension $m+1$ and $Q_x\subset<Q_x>$ is a quadric hypersurface);

\item[ii)] there exists a point $z \in X$ such that for $x \in X$
general, the quadric $Q_x$ passes through $z$.
\end{enumerate}
\medskip

We will say such a variety is {\em  swept out  by $m$-dimensional
hyperquadrics passing through} {\bf $z\in X$}. For example, a LQEL
manifold with secant defect $\delta >0$ is swept out by
$\delta$-dimensional hyperquadrics passing through a point. By
Lemma \ref{cclem}, a smooth variety is conically connected if and
only if it is swept out by a 1-dimensional hyperquadrics passing
through a point.

\begin{Lem} \label{CC}
The secant defect $\delta$ of  a  variety $X\subset\p^N$ swept out
by $m$-dimensional hyperquadrics passing through a point $z\in X$
satisfies $\delta\geq m$.
\end{Lem}
\begin{proof}
Let $\Hilb^{conic, z}(X)$ be the Hilbert scheme of conics in $X$
passing  through $z$ and  let $W_1, \cdots, W_k$  be its
irreducible components.  If $z$ is a singular point of $Q_x$ for
$x\in X$ general, then the line $<z,x>$ would be contained in $X$
and by Lemma \ref{cclem} $X\subset\p^N$ would be degenerated
contrary to our assumption. Thus for general $x\in X$, $z$ is a
smooth point of $Q_x$ and
 the Hilbert scheme
$\Hilb^{conic, z}(Q_x)$ is irreducible, so that  there exists some
$i \in \{1, \cdots, k\}$ such that $\Hilb^{conic, z}(Q_x) \subset
W_i.$ This implies that there exists a component $W:=W_{i_0}$
containing $\Hilb^{conic, z}(Q_x)$ for $x \in X$ general. This
gives the dimension estimate:
\begin{equation} \label{dim}
\dim W \geq n+m-2.
\end{equation}

 Reasoning as in  the proof of Lemma \ref{cclem},  if we take a general point $x
\in X$ and an irreducible conic $[C] \subset Q_x$ joining $x$ and
$z$, then we can suppose that  $N_{C|X}$ is ample. Thus $W$ is
smooth at the point $[C]$ and
$$\dim(W) = \dim H^0(C, N_{C|X} \otimes \0_C(-z)) = -K_X \cdot C -
2.$$ Combining with \eqref{dim}, we obtain
\begin{equation}\label{dim2}
 -K_X \cdot
C \geq n+m.\end{equation}

 By Lemma \ref{cclem},  $X\subset\p^N$ is conically connected
so that Proposition \ref{LQELCC} gives $n + \delta \geq -K_X \cdot
C \geq n+m$, yielding $\delta \geq m$.
\end{proof}
\medskip

An immediate consequence of this lemma and Prop. \ref{LQELCC} is
the following result.
\medskip

\begin{Cor}\label{Cor}
If $m \geq 3$, then $X\subset\p^N$ is a Fano variety with $\Pic(X) = \zit\langle\O(1)\rangle$ and
the index $i(X)$ satisfies $$ \frac{n+\delta}{2} \geq i(X) \geq
\frac{n+m}{2}.$$
\end{Cor}
\medskip

Recall that for a general point $x \in X$, the variety $Y_x$ is
the Hilbert scheme of lines on $X$ passing through $x$.

\begin{Lem}\label{Yx}
 Assume that $m \geq 3$. Then $Y_x$ is smooth irreducible of
dimension $i(X)-2$. If moreover $m > n/3$, then $Y_x$ is
non-degenerate and $SY_x = \pit^{n-1}$.
\end{Lem}
\begin{proof}
 Corollary \ref{Cor} yields $i(X)\geq (n+m)/2\geq (n+3)/2$. By part (i)
of Prop. \ref{largeindex} we deduce that $Y_x\subset\p^{n-1}$ is
not empty and irreducible. If $l_x$ is a line through $x$, then
$\dim(Y_x)=H^0(N_{l_x|X})=-K_X\cdot l_x-2=i(X)-2$. The last part
follows form (iii) of Prop. \ref{largeindex}.
\end{proof}

In the sequel we shall  use the following simple remark.
\medskip

\begin{Lem}\label{Yxquadric}
 Assume  $n \geq 2$ and $\delta\geq 1$. If $Y_x\subset\p^{n-1}$ is a
 non-degenerate hypersurface, then $Y_x\subset\p^{n-1}$ is a smooth quadric
 hypersurface and $X\subset\p^{n+1}$ is a smooth quadric hypersurface.
 \end{Lem}
\begin{proof} Since $\delta\geq 1$, the second fundamental form $|II_{x,X}|\subseteq|\O_{\p^{n-1}}(2)|$
is a linear system of quadrics  of dimension $N-n-1$ (see for
example \cite[Thm. 2.3 (1)]{Rus}).
 Since $Y_x\subset\p^{n-1}$ is
 contained in the base locus scheme of $|II_{x,X}|$ and since it is a non-degenerate hypersurface,
we obatain that $Y_x\subset\p^{n-1}$ is a quadric hypersurface and that $N=n+1$, i.e. $X\subset\p^{n+1}$
is a  hypersurface. Let $l_x\subset\ X$ be a line passing through $x$.
Reasoning as in the proof of Lemma \ref{Yx} we get, by adjunction,
$$n-2=\dim(Y_x)=-K_X\cdot l_x-2=-(\deg(X)-n-2)-2,$$
that is $\deg(X)=2$ as claimed.
\end{proof}
\medskip

We now prove  a substantial  improvement of the Main Theorem 0.2
of \cite{KS}, where the same claim is proved under the  stronger
assumption that a general hyperquadric is smooth and that $m \geq
3n/5+1$ if $n=5, 6$ or $ 10$ and $m \geq 3n/5$ otherwise.
\medskip

\begin{Thm}\label{main}
Let $X^n \subsetneq \p^N$ be a smooth non-degenerate variety, which is swept out
by $m$-dimensional hyperquadrics passing through a point.
If $m > [n/2]+1$, then $N=n+1$ and
$X$ is itself a smooth hyperquadric.
\end{Thm}
\begin{proof}
The condition $m > [n/2]+1$ implies   $m \geq 3$. By Lemma
\ref{Yx} we know that $Y_x\subset \pit^{n-1}$ is a smooth
non-degenerate variety.  Reasoning as in the proof of Lemma
\ref{CC} we can suppose that,
 for $x\in X$ general,
$z$ is a smooth point of $Q_x$, so  that lines on the quadric
$Q_x$ passing through $z$ are parameterized by an
$(m-2)$-dimensional quadric hypersurface
$\tilde{Q}_x\subset\p^{n-1}$. Clearly   $ \tilde{Q}_x \subset Y_x
$.  By assumption, $m-2 > [(n-2)/2]$, so $Y_x\subset\p^{n-1}$
contains a high dimensional variety which is a hypersurface in its
linear span in $\p^{n-1}$. Then \cite[Corollary I.2.20]{Zak}
implies that $Y_x\subset\p^{n-1}$ is itself a hypersurface and the
conclusion now follows from Lemma \ref{Yxquadric}.
\end{proof}

The following corollary is analogue to results in \cite{Ein},
\cite{Wis} and \cite{Sat}, where they considered linear subspaces
instead of quadric subvarieties.
\begin{Cor}
Let $X \subset \pit^N$ be a smooth non-degenerate variety of
dimension $n$ and $Q \subset X$ a smooth quadric subvariety of
dimension $m$ whose normal bundle $N_{Q|X}$ is isomorphic to
$\0_Q(1)^{\oplus n-m}.$  If $m > [n/2]+1$, then $X$ is a
hyperquadric.
\end{Cor}
\begin{proof}
Let $\I_q$ be the ideal sheaf of a point $q \in Q$. By the exact
sequence $0 \to N_{Q|X} \otimes \I_q \to N_{Q|X} \to N_{Q|X, q}
\to 0$, we get $H^1(Q, N_{Q|X} \otimes \I_q)=0$, since $N_{Q|X}
\simeq \0_Q(1)^{\oplus n-m}$ is globally generated and $H^1(Q,
\0_Q(1))=0$. Similarly, since $T_Q$ is globally generated  and
$H^1(Q, T_Q)=0$, we obtain $H^1(Q, T_Q \otimes \I_q)=0$.  Note
that the following sequence is exact: $$ 0 \to T_Q \otimes \I_q
\to T_X|_Q \otimes \I_q \to N_{Q|X} \otimes \I_q \to 0.$$ The long
exact sequence of cohomology gives $H^1(Q, T_X|_Q \otimes \I_q)=0$
and the following sequence is exact:
\begin{equation} \label{exa}
0 \to H^0(Q, T_Q \otimes \I_q) \to H^0(Q, T_X|_Q \otimes \I_q) \to
H^0(Q, N_{Q|X} \otimes \I_q) \to 0.
\end{equation}
Let $\Mor(Q, X;q)$ be the variety parameterizing morphisms from
$Q$ to $X$ fixing the point $q$. Then it is smooth at $\iota: Q
\to X$, the natural inclusion.

Consider the  evaluation  map  $ev: Q \times \Mor(Q, X; q) \to X$.
Take a point $p \in Q-\{q\}$.
The tangent map to $ev$ at point $(p, \iota)$ is $$T_p Q \oplus
H^0(Q, T_X|_Q \otimes \I_q) \to T_{p, X}$$  given by  $$(u,
\sigma) \mapsto T_p\iota(u) + \sigma(p)= u + \sigma(p).$$
Thus the image contains $T_pQ$. To show it is surjective, we just need to
show that the composition map  $H^0(Q, T_X|_Q \otimes \I_q) \to T_{p, X} \to N_{p,Q|X}, \sigma
\mapsto [\sigma(p)]$ is surjective.
 By the exact sequence \eqref{exa}, it is enough to show that  $H^0(Q, N_{Q|X} \otimes \I_q) \to N_{p, Q|X}$
is surjective, i.e. $H^0(Q, \0_Q(1) \otimes \I_q) \to k_p$ is surjective. This is immediate from
the very ampleness of $\0_Q(1)$.

 In particular, this implies that the map $ev$ is smooth
at points $(p, \iota)$ for $p \neq q$. Thus the deformations of
$Q$ while fixing $q$ dominant $X$. Now we can apply the precedent
theorem to conclude.
\end{proof}

Next we will consider the case $m=[n/2]+1$ with $n \geq 3$.
\begin{Prop}
If $N \geq 3n/2$ and $m = [n/2]+1$ with $n \geq 3$, then $X$ is
projectively isomorphic to one of the following:

i) the Segre 3-fold $\pit^1 \times \pit^2 \subset \pit^5$;

 ii) the
Pl\"ucker embedding $\mathbb{G}(1, 4) \subset \pit^9$;

 iii) the
10-dimensional spinor variety $S^{10} \subset \pit^{15}$;

 iv) a
general hyperplane section of ii) or iii).
\end{Prop}
\begin{proof}
As $\delta \geq m > n/2$,  by Zak's linear normality theorem
(\cite{Zak}), we have $SX = \pit^N$. Thus $\delta = 2n+1 -N \geq
[n/2]+1$, which gives that $N \leq 2n-[n/2].$ By hypothesis, $N
\geq 3n/2$, thus $N = 2n-[n/2]$, which gives  $\delta=m$. As a
consequence,  $-K_X \cdot C = n+ \delta$ for a generic conic $C$,
which implies that $X$ is a LQEL manifold of type $m$ by
\cite[Prop. 3.2]{IR2}. Now the claim is given  by the
classification result in \cite[Cor. 3.1]{Rus}.
\end{proof}
\begin{Rque}\upshape
Here we give an outline of an approach to classify such varieties
$X$ with $m=[n/2]+1$, based on Hartshorne's conjecture.
 We may assume that $Y_x$ is not
a hypersurface in $\pit^{n-1}$, i. e. $n-2-\dim Y_x \geq 1$. By
the proof of Prop. I.2.16 [Zak], for any hyperplane $H \subset
\pit^{n-1}$ containing the linear span of $\tilde{Q}_x$ and $T_yY$
for some $y \in Y$, $H$ is tangent to $Y_x$ along some variety $Z
\subset \tilde{Q}_x$. The dimension of $Z$ is bounded by
$$
n-2- \dim Y_x \geq \dim Z \geq 2(m-2) - \dim Y_x=2[n/2]-2-\dim
Y_x.
$$

Consider the Gauss map: $\gamma_{n-2}:  \mathcal{P}_{n-2} \to
(\pit^{n-1})^*$ (cf. I.2 [Zak]). By definition,
$\gamma_{n-2}^{-1}(H)$ contains the variety $Z \times \{ H\}$.
 If $2n-2 < 3 i(X)-6$, then Harsthorne's conjecture implies that
$Y_x \subset \pit^{n-1}$ is a complete intersection. By Prop. I.
2.10 [Zak], the map $\gamma_{n-2}$ is finite, which gives $\dim Z
= 0$. We deduce that $n$ is odd and $i(X)=n-1$, so $X$ is a smooth
Del Pezzo varieties, which have been completely classified.

Thus we may assume $2n-2 \geq 3 i(X)-6$, which gives $2n+4 \geq
3/2(n+m) = 3/2(n+[n/2]+1)$. This implies that $n \leq 11$ or
$n=13$. When $n \leq 11$, we obtain that $i(X) \geq n-2$, thus $X$
is a Fano variety with $Pic \simeq \zit$ and of coindex at most 3,
i. e. $X$ is either a Del Pezzo variety or a Mukai variety. The
case $n=13$ with $i(X) = 10$ requires a more detailed study.
\end{Rque}

\section{LQEL-manifolds with  large secant defects}

The idea contained in the proof of Theorem \ref{main} can be
combined with the Divisibility Theorem of \cite{Rus}, obtaining
new constraints for the existence of $LQEL$-manifold with large
secant defects.
\medskip

Let $X\subset\p^N$ be a $LQEL$-manifold of type $\delta \geq 2k+1$. We define
inductively a sequence of smooth varieties: $Y_1: = Y_x\subset\p^{n-1}$ and let
$Y_{j+1}\subset\p^{\dim(Y_j)-1}$ be  the Hilbert scheme of lines on $Y_j$ passing through
a  general point of it, for $k-1 \geq j \geq 1$. By the previous
theorem, we know that $Y_j\subset\p^{\dim(Y_{j-1})-1}$ is a $LQEL$-manifold of type
$\delta-2j$ with $SY_j=\p^{\dim(Y_{j-1})-1}$. Furthermore for $j \leq k-1$, $Y_j\subset\p^{\dim(Y_{j-1})-1}$ is a Fano variety
with $\Pic(Y_j) = \zit \langle \0(1) \rangle$ (\cite{Rus}).
Let $i_j$ be the index of $Y_j$ and $i_0=(n+\delta)/2$ the
index of $X$. The following lemma can also be deduced from the Divisibility Theorem
cited above.
\medskip

\begin{Lem} \label{index}
$$ i_j = \cfrac{n-\delta}{2^{j+1}} + \delta - 2j,  \  0 \leq j \leq k-1.$$
\end{Lem}
\begin{proof}
By Theorem \ref{induction}, we have $$2 i_j = \dim  Y_{j} +
\delta(Y_{j}) = i_{j-1}-2 + \delta - 2j,$$ which gives
$2(i_j+2j-\delta) = i_{j-1}+2(j-1) - \delta.$ We deduce that $
i_j+2j-\delta = (i_0 - \delta)/2^j$, concluding the proof.
\end{proof}
\medskip

\begin{Thm}\label{quadricLQEL}
Let $X\subset\p^N$ be an $n$-dimensional $LQEL$-manifold of type $\delta$. If
$$\delta > 2[\log_2 n]+2 \ {\text or} \  \delta > \min_{k \in \nit} \{ \cfrac{n}{2^{k-1}+1} + \cfrac{2^{k}
k}{2^{k-1}+1} \},$$ then $N=n+1$ and $X\subset\p^{n+1}$ is a quadric hypersurface.
\end{Thm}
\begin{proof}
If $\delta > 2[\log_2 n]+2$, then $n < 2^r$, where $r
=[(\delta-1)/2]$. By Theorem \ref{induction}, $2^r$ divides
$n-\delta$. This is possible only if $\delta = n$. Thus $X$ is a
hyperquadric. Now assume we have the second inequality.  Note that
for a fixed $n$, the minimum $\min_{k \in \nit} \{
\cfrac{n}{2^{k-1}+1} + \cfrac{2^{k} k}{2^{k-1}+1} \}$ is achieved
for some $k \leq n/2$, so we may assume that for some $k \leq
n/2$, we have $\delta > \cfrac{n}{2^{k-1}+1} + \cfrac{2^{k}
k}{2^{k-1}+1}  = 2k + \cfrac{n-2k}{2^{k-1}+1} \geq 2k,$ so that
 $\delta \geq 2k+1$.  Now we can consider the variety
$Y_{k} \subset \pit^{\dim Y_{k-1} - 1}$. Note that $\dim Y_{k} =
i(Y_{k-1})-2$ and
$$
\dim Y_{k-1} = 2 i_{k-1} - \delta(Y_{k-1}) =
\cfrac{n-\delta}{2^{k-1}} + \delta - 2k +2.
$$

On the other hand, $Y_k\subset\p^{\dim(Y_{k-1})-1}$ is non-degenerate and it contains a
hyperquadric of dimension $\delta - 2k$, which is strictly bigger
than $(\dim Y_{k-1} -2)/2$ under our assumption on $\delta$. Now
\cite[Corollary I.2.20]{Zak} implies that  $Y_k\subset\p^{\dim(Y_{k-1})-1}$ is a hypersurface. Since it is
a non-degenerate hypersurface by Theorem \ref{induction}, a repeated application of Lemma
\ref{Yxquadric} yields the conclusion.
\end{proof}
\medskip

We now state a sharper Linearly Normality Bound for
$LQEL$-manifolds, see \cite[II.2.17]{Zak}. Moreover, in \cite[Cor.
3.1, Cor. 3.2]{Rus}   Russo has classified $n$-dimensional
$LQEL$-manifolds of type $\delta \geq  n/2$. Combining these
results with  the bound on $\delta$ in the Theorem
\ref{quadricLQEL} we are able to classify  the extremas cases of
the bounds.
\medskip

\begin{Cor}\label{boundelta} Let $X\subset\p^N$ be a $LQEL$-manifold of type $\delta$, not a quadric
hypersurface. Then $$\delta\leq\min_{k \in \nit} \{ \cfrac{n}{2^{k-1}+1} + \cfrac{2^{k}
k}{2^{k-1}+1} \}\leq
\cfrac{n+8}{3}$$
and $$N\geq\dim(SX)\geq 2n+1-\min_{k \in \nit} \{ \cfrac{n}{2^{k-1}+1} + \cfrac{2^{k}
k}{2^{k-1}+1} \}\geq \frac{5}{3}(n-1).$$

Furthermore $\delta=\frac{n+8}{3}$ if and only if $X\subset\p^N$ is projectively equivalent
to one of the following:
\begin{enumerate}
\item[i)] a smooth 4-dimensional quadric hypersurface $X\subset\p^5$;

\item[ii)] the $10$-dimensional spinor variety
$S^{10}\subset\p^{15}$;

\item[iii)] the $E_6$-variety $X\subset\p^{26}$ or one of its isomorphic projection in $\p^{25}$;

\item[iv)] a $16$-dimensional linearly normal rational variety
$X\subset\p^{25}$, which is a Fano variety of index $12$ with
$SX=\p^{25}$, dual defect $\defect(X)=0$ and such that the base locus
scheme $C_x\subset\p^{15}$  of $|II_{x,X}|$ is
the union of $10$-dimensional spinor variety $S^{10}\subset\p^{15}$
with  $C_pS^{10}\simeq\p^{7}$, $p\in\p^{15}\setminus S^{10}$.
\end{enumerate}
\end{Cor}
\begin{proof} We shall prove only the second part. If $\delta=\frac{n+8}{3}$, then
$n-\delta=\frac{2n-8}{3}$. Suppose $\delta=2r_X+1$, so that $n-\delta=\frac{12r_X-18}{3}$.
By Theorem \ref{induction} we deduce that $2^{r_X}$ should divide $4r_X-6$, which is not possible.

Suppose now $\delta=2r_X+2$, so that
$n-\delta=\frac{12r_X-12}{3}=4(r_X-1)$. Since $2^{r_X}$ has to
divide $4(r_X-1)$, we get $r_X=1,2,3$ and, respectively,  $n=4,
10, 16$ with $\delta=4, 6$, respectively 8. The conclusion follows
from \cite[Cor 3.1 and Cor. 3.2]{Rus}.
\end{proof}
\medskip

Let us observe that Lazarsfeld and Van de Ven posed the question if for an irreducible
smooth projective non-degenerate $n$-dimensional variety $X\subset\p^N$ with $SX\subsetneq\p^N$
the secant defect is bounded, see \cite{LVdV}. This question was motivated by the
fact that for the known examples we have $\delta(X)\leq 8$, the bound being attained for the
sixteen dimensional Cartan variety $E_6\subset\p^{26}$, which is a $LQEL$-variety
of type $\delta=8$. Based on these remarks and on the above results one could naturally formulate
the following problem.
\medskip

{\em Question:} Is a $LQEL$-manifold  $X\subset\p^N$ with $\delta
>8$   a smooth quadric hypersurface?
\medskip

\section{ Acknowledgements} It is my pleasure to thank Prof. F. Russo
for introducing this problem to me during Pragmatic 2006, for his
insight that a simple proof of the main result in \cite{KS} in the
spirit of \cite{IR2}, \cite{Rus} is possible and  also for the
numerous discussions, suggestions and corrections that lead to the
final form of this paper. I'm grateful to Prof. F. L. Zak for his
constant help and suggestions, especially for drawing my attention
to Prop. I.2.16 \cite{Zak}, which is one of the key points of the
proof. I would like to thank J. Caravantes for helpful discussions
on this problem and the organizers of Pragmatic 2006 for the
hospitality in Catania. I'm grateful to the referee for very
helpful suggestions, which greatly improved the present note.

\quad \\
C.N.R.S., Labo. J. Leray, Facult\'e des sciences, Universit\'e de NANTES \\
2, Rue de la Houssini\`ere,  BP 92208,
             F-44322 Nantes Cedex 03 - France \\
fu@math.univ-nantes.fr

\end{document}